\documentclass{amsart}

\usepackage{amssymb,amsfonts, amsmath, amsthm}
\usepackage[all,arc]{xy}
\usepackage{enumerate}
\usepackage{mathrsfs}
\usepackage{mathtools}
\usepackage[mathscr]{euscript}
\usepackage{array}
\usepackage{appendix}

\usepackage{tikz}
\usepackage{tikz-cd}
\usepackage{mathpazo} 
\usepackage{textcomp}
\usepackage{bbold}
\usepackage[bbgreekl]{mathbbol}
\usepackage{enumitem}

\usepackage[pagebackref, colorlinks=true, linkcolor=blue, citecolor = red]{hyperref}
\renewcommand*{\backref}[1]{}
\renewcommand*{\backrefalt}[4]{({%
		\ifcase #1 Not cited.%
		\or On p.~#2%
		\else On pp.~#2%
		\fi%
	})}

\usepackage[nameinlink,capitalise,noabbrev]{cleveref}
\crefname{subsection}{Subsection}{Subsection}

\usetikzlibrary{patterns}

\usepackage{geometry}
\DeclareMathAlphabet{\mathbbe}{U}{bbold}{m}{n}

\def\DDelta{{\mathbbe{\Delta}}}
\newcommand{\DD}{\DDelta}

\newcommand{\C}{\mathscr{C}}
\newcommand{\cC}{\mathcal{C}}
\newcommand{\D}{\mathscr{D}}
\newcommand{\cF}{\mathcal{F}}
\newcommand{\cG}{\mathcal{G}}
\newcommand{\M}{\mathscr{M}}
\newcommand{\cM}{\mathcal{M}}
\newcommand{\cN}{\mathcal{N}}
\newcommand{\bN}{\mathbb{N}}
\newcommand{\cR}{\mathcal{R}}

\newcommand{\U}{\mathscr{U}}
\newcommand{\V}{\mathscr{V}}
\newcommand{\cW}{\mathcal{W}}

\newcommand{\set}{\mathscr{S}\mathrm{et}}
\newcommand{\sset}{s\set}
\newcommand{\cat}{\mathscr{C}\mathrm{at}}
\newcommand{\Subfib}{\mathscr{S}\mathrm{ub}^{\mathrm{fib}}}
\newcommand{\Fun}{\mathrm{Fun}}
\newcommand{\Map}{\mathrm{Map}}
\newcommand{\Hom}{\mathrm{Hom}}
\newcommand{\Ho}{\mathrm{Ho}}
\newcommand{\id}{\mathrm{id}}
\newcommand{\Shv}{\mathscr{S}\mathrm{hv}}
\newcommand{\colim}{\mathrm{colim}}
\newcommand{\comp}{\mathrm{comp}}
\newcommand{\Grp}{\mathscr{G}\mathrm{rp}}
\newcommand{\Ab}{\mathscr{A}\mathrm{b}}
\newcommand{\Top}{\mathscr{T}\mathrm{op}}
\newcommand{\ev}{\mathrm{ev}}
\newcommand{\Fil}{\mathscr{F}\mathrm{il}}

\newtheorem{theorem}[equation]{Theorem}
\newtheorem{lemma}[equation]{Lemma}
\newtheorem{proposition}[equation]{Proposition}
\newtheorem{corollary}[equation]{Corollary}

\theoremstyle{definition}
\newtheorem{definition}[equation]{Definition}
\newtheorem{example}[equation]{Example}

\theoremstyle{remark}
\newtheorem{remark}[equation]{Remark}

\newtheorem{notation}[equation]{Notation}
\newtheorem{warning}[equation]{Warning}

\numberwithin{equation}{section}

\title{Filter Quotient Model Structures}

\date{March 2026}

\author{Nima Rasekh}
\address{Institut f{\"u}r Mathematik und Informatik, Universit{\"a}t Greifswald, Greifswald, Germany}
\email{nima.rasekh@uni-greifswald.de}


\begin{document}

\begin{abstract}
	The filter quotient construction is a particular instance of a filtered colimit of categories. It has primarily been considered in the context of categorical logic, where it has been used effectively to construct non-trivial models, for example new models of set theory.

	In this work we prove that given a model category and a suitable notion of filter of subterminal objects, the filter quotient construction will preserve the model structure. We also show that this new model structure inherits certain important properties (such as being simplicial or proper), but not all (such as being cofibrantly generated). Finally, we show it is compatible with the construction of filter quotient $\infty$-categories \cite{rasekh2021filterquotient}.
\end{abstract}

\maketitle

\section{Introduction}

\subsection{Overview: Model Categories and Type Theories}
This paper introduces a new class of model categories, \emph{filter quotient model categories}, 	whose construction does not require many common technical assumptions, such as local presentability of the underlying category or cofibrant generation, and has concrete applications in type theory. In an effort to make the text accessible to a broad range of readers with a variety of backgrounds, the introduction aims to provide relevant context for the important ideas. This includes an overview of model category theory (\cref{subsec:modelcats,subsec:newmodelcats,subsec:modelcatspresentable}) and an overview of type theories and their models (\cref{subsec:typetheory,subsec:modelcatstypetheories}), culminating in a review of the main result (\cref{subsec:filterquotientmodelstructures}). 

\subsection{Model Categories in Homotopy Theory} \label{subsec:modelcats}
Homotopy theory arises whenever we study mathematical objects that exhibit a non-trivial notion of equivalence, or ``sameness''. This first manifested in the study of topological spaces, with the notion of (weak) homotopy equivalence of topological spaces, which in general is weaker than a homeomorphism. From there it generalized to quasi-isomorphisms of chain complexes, and then to many other mathematical objects in algebra and geometry.

Applying a mathematical lens to this situation suggests the need for a formal structure that can capture such mathematical objects along with their equivalences. This was first achieved by Quillen. In \cite{quillen1967modelcats}, he defined \emph{model categories} as an abstract category along with a chosen class of morphisms that are ``weak equivalences'', along with further data that permits an effective handling of this data.

Since then, many mathematicians have used the model categorical framework to study equivalences of interest. Examples include the study of stable homotopy theory and spectra \cite{ekmm1995stable}, the study of motivic homotopy theory \cite{voevodsky1998motivicorigin}, the study of $\infty$-category theory \cite{rezk2001css}, and the study of type theories \cite{kapulkinlumsdaine2021kanunivalent}. In all of these cases the structure of a model category has non-trivially contributed towards a better understanding of the mathematical objects and the equivalences in question.

\subsection{New Model Categories (from	Old)} \label{subsec:newmodelcats}
Unfortunately, the model categorical framework comes at a price. A model category entails substantial data and numerous conditions, making it difficult to construct new model structures just out of a category and a desired class of weak equivalences. This has resulted in a variety of methods that permit obtaining model structures from a specified collection of data.

\begin{enumerate}[leftmargin=*]
	\item \textbf{Cisinski model structures:} Given a suitable collection of monomorphisms, denoted $\mathrm{An}$, in a presheaf category $\Fun(\C^{op},\set)$, there is a model structure on $\Fun(\C^{op},\set)$, called the \emph{Cisinski model structure}, whose cofibrations are monomorphisms and fibrant objects satisfy the right lifting property with respect to $\mathrm{An}$ \cite{cisinski2006cisinskimodelstructure}. More generally, we have analogous results for all Grothendieck topoi \cite{cisinski2002homotopytopos}.
	\item \textbf{Induced (transferred) model categories:} Given a suitable model category and a left (or right) adjoint functor $F\colon \cM \to \C$, if $\cM, F, \C$ satisfy certain conditions, we can construct a model structure on $\C$, that is ``induced'' or ``transferred" from $\cM$, meaning that the weak equivalences and cofibrations (or fibrations) of $\C$ are determined via the functor $F$. Induced model structures have been used for decades \cite{crans1995quillen,goerssschemmerhorn2007model}, however, recently we have seen various refinements generalizing previous results \cite{bayehhesskarpovakedziorekriehlshipley2015leftinduced,hesskedziorekriehlshipley2017inducedmodel,garnerkedziorekriehl2020lifting}.
	\item \textbf{Functor model categories:} Given a suitable model category $\cM$ and an arbitrary category $\C$, we can use a variation of inducing model structures to construct two model structures on the functor category $\Fun(\C,\cM)$, known as the \emph{projective model structure} and the \emph{injective model structure} \cite{heller1982functorcategory}. In recent years we have seen generalizations of these methods to a variety of settings, such as enriched model structures \cite{moser2019injective}.
	\item \textbf{Bousfield localizations:} Given a suitable model category $\cM$ and a choice of cofibrations (or fibrations) one can define a \emph{left Bousfield localization} (or \emph{right Bousfield localization}) on the same category \cite{bousfield1975localization}. This approach has since been generalized to the enriched case \cite{barwick2010bousfieldlocalization}, and the monoidal case \cite{white2022bousfieldlocalizations}.
	\item \textbf{Reedy model categories:} Given a general model category $\cM$ and a \emph{Reedy category} $\cR$, we can construct the \emph{Reedy model structure} on the functor category $\Fun(\cR,\cM)$ \cite{reedy1974modelstructure}. This approach has since been generalized to the enriched setting \cite{angeltveit2008enrichedreedy,ghazelkadhi2019reedy}, and also to more general diagram categories, known as \emph{generalized Reedy categories} \cite{cisinski2006cisinskimodelstructure,bergermoerdijk2011extension}.
	\item \textbf{Finite Posets:} Given a finite poset, such as $[n]$, we can use methods from \emph{homotopical combinatorics} and counting methods from the theory of operads, to explicitly characterize and obtain model structures \cite{drozzakharevich2015simplehomotopy,balchinormsbyosornoroitzheim2023modelstructures}.
\end{enumerate}

All of these constructions fall into two broad classes. In Item (1) - (4) we have a large category with a manageable collection of data, such as a set of generating monomorphisms, or a set of generators (local presentability), or a set of generating (trivial) cofibrations (cofibrant generation), which permit reducing arguments to these generators. In Item (5) - (6) we instead have finite inductive conditions, such as a Reedy category or a finite poset, which permit constructing the model structure in a step-by-step or combinatorial manner.

This leaves us with the cases where we have a category that is simultaneously too large to satisfy any finiteness condition, but is also not ``small generated" in the sense that it satisfies accessibility, local presentability or cofibrant generation. As of now the construction of model structures in those cases has largely remained unexplored. The aim of this paper is to present one such method.

\subsection{From Model Categories to \texorpdfstring{$\infty$}{oo}-Categories and back} \label{subsec:modelcatspresentable}
Before proceeding to a new method to construct model structures, it is instructive to first motivate this endeavor. \emph{$\infty$-category theory} arose as another method to study homotopical mathematics. This has been made precise via \emph{models of $\infty$-categories}, all of which are equivalent in a suitable way and generalize classical categories. This includes categorical structures with a weak notion of composition, such as \emph{quasi-categories} \cite{joyal2008notes}, as well as categories with an abstract notion of weak equivalences, such as \emph{relative categories} \cite{barwickkan2012relativecategory}. In recent decades we have seen significant effort developing category theory in the context of $\infty$-categories as a suitable setting for homotopy theory, paralleling and often surpassing analogous developments in model category theory \cite{lurie2009htt,riehlverity2022elements}.

Model categories and $\infty$-categories are deeply intertwined. Indeed, every model category comes with an \emph{underlying $\infty$-category}. However, not every $\infty$-category can be recovered as the underlying $\infty$-category of a model category, as any such $\infty$-category would necessarily have some (co)limits. However, certain applications of $\infty$-category theory significantly benefit from working with $\infty$-categories that come from model categories. Thus, it is natural to wonder which $\infty$-categories can be obtained this way. 

One major step in this direction was taken by Dugger, who proved that every \emph{presentable $\infty$-category} can be obtained as the underlying $\infty$-category of a \emph{combinatorial model category}, i.e.,~a cofibrantly generated model category whose underlying category is locally presentable \cite{dugger2001universal}. The construction of such model categories precisely relies on the techniques outlined in \cref{subsec:newmodelcats}, that are designed for categories with ``small generating data'', which is an essential aspect of both presentable $\infty$-categories and combinatorial model categories. 

If a given $\infty$-category does not exhibit such properties (e.g.,~does not have small colimits or is not accessible), then the aforementioned techniques cannot be used in an analogous manner and we need new methods to construct model structures. For example, \emph{filter quotient $\infty$-categories} \cite{rasekh2021filterquotient} are usually not presentable and even lack countable colimits, yet would benefit from a model categorical presentation, due to their anticipated connection with models of type theories, which we now explore in more detail.

\subsection{Enter Type Theory} \label{subsec:typetheory}
Before we can understand models of type theories and their connection to model categories, we first briefly review type theory itself. For a more detailed introduction see \cite{hottbook2013,riehl2024inftytopos,rijke2025introductionhott}. \emph{Type theory} is a foundation for mathematics consisting of foundational objects called \emph{types} $X$, which are analogous to sets. The elements therein are called \emph{terms} $x\colon X$ and are analogous to elements of a set. Beyond that, there are a variety of type constructors	that permit constructing new types from existing ones. These include type constructors that have analogues in set theory, such as \emph{product types} $X \times Y$, which is analogous to the Cartesian product, or \emph{function types} $X \to Y$, which is analogous to the set of functions from $X$ to $Y$. It also includes certain type constructors that have no analogues in set theory and are used for logical reasoning in type theory, such as the \emph{identity type} $ x =_X y$, for two terms $x,y \colon X$. Unlike in set theory, where sets and formulas about sets live in two different worlds, these \emph{logical type constructors} allow constructing types that are themselves statements about types. A \emph{proof}, for example that two terms $x,y:X$ are equal, then is a term of such a type i.e.,~a term $p\colon x =_X y$. Such proofs must be obtained following the various syntactic rules that govern the behavior of these types. Thus proofs in type theory have a \emph{syntactic} nature, in that proofs are obtained via formal rules that govern the behavior of the types. 

This syntactic nature of type theory makes it particularly suitable for \emph{formalization of mathematics}. Formalization is a process in which mathematical statements are entered into a computer via suitable programming languages, such as Lean or Rocq, and formally verified using the rules of the programming language. Due to their common syntactic nature, many of these programming languages are based on type theoretical foundations \cite{coquandhuet1986coc}. As a result a proof via a proof assistant, such as Lean, really corresponds to a formal proof in the underlying type theory. This is complicated by the fact that different proof assistants rely on different type theories, meaning the same syntactic proof in different proof assistants might correspond to different mathematical statements.

In recent years, and particularly with advances in AI, proof assistants have become increasingly popular, resulting in many successful and ongoing formalization projects \cite{favoniafinsterlicatalumsdaine2016blakersmassey,ljungstrommortberg2023pi4s3,vandoornmassotnash2023hprinciple,kudasovriehlweinberger2024yoneda}. This, however, poses its own challenge, how can we be certain that the formalized result coincides with the intended mathematical statement? Here the solution is to relate type theory to more traditional mathematical foundations, such as set theory, via \emph{models of type theory}.

\subsection{From Type Theories to Model Categories} \label{subsec:modelcatstypetheories}
For a given type theory, a \emph{model} in a given foundation is a precise interpretation of the types, terms and constructors	of the type theory in the given foundation. For example, exhibiting sets as a model of type theory, means interpreting types as sets, terms as elements of sets, product types as Cartesian products of sets, function types as sets of functions, and identity types as equality of elements. Constructing such a model for a type theory is known as a	\emph{soundness} result, as it guarantees that all type theoretical statements can be interpreted meaningfully in some foundation. This has resulted in significant efforts establishing various soundness results for type theories \cite{seely1983hyperdoctrines,seely1984locallycartesian,lambekscott1988higherorderlogic,hofmann1995lccc,curiengarnerhofmann2014categorical,clairambaultdybjer2014biequivalence}.

This development in particular includes models for \emph{homotopy type theory}, which is a type theory with a more intricate notion of identity, that is particularly suitable for the formalization of homotopical statements \cite{hottbook2013}. There was an early realization that if an $\infty$-category can be realized by a suitable model category, then it provides a model for homotopy type theory \cite{kapulkinlumsdaine2021kanunivalent}\footnote{It is an unfortunate historical accident that the two occurrences of the word ``model'' in ``... model category, then it provides a model for homotopy type theory.'' are unrelated.}, hence giving us soundness results in this setting. This has resulted in a series of papers constructing more and more intricate models for homotopy type theory  via model categories \cite{arndtkapulkin2011modelstypetheory,shulman2015homotopycanonicity,lumsdaineshulman2020goodexcellent,shulman2019inftytoposunivalent}. 

While this development has been a tremendous success, all models of homotopy type theory constructed this way have one thing in common: they are combinatorial\footnote{In fact a \emph{type-theoretic model topos}, which is a model categorical model of homotopy type theory due to Shulman includes being combinatorial as part of its definition \cite{shulman2019inftytoposunivalent}.}. However, ample evidence from other type theories suggests that homotopy type theory should indeed exhibit additional models, in particular those coming from filter quotients \cite{adelmanjohnstone1982serreclasses}. We hence need new methods to construct model structures without requiring any ``small generation'' assumption, and that in particular exhibit filter quotient $\infty$-categories as an underlying model category.

\subsection{Filter Quotient Model Structures} \label{subsec:filterquotientmodelstructures}
In the previous subsection, we discussed how established methods for constructing model structures commonly rely on finiteness assumptions or some ``small generation'' assumptions (\cref{subsec:newmodelcats}). However, motivated by the study of type theories and their models, we are interested in constructing model categories in a context where these assumptions do not hold (\cref{subsec:modelcatstypetheories}). The aim of this paper is to introduce a novel technique for constructing new model structures from a given model structure that uses the \emph{filter quotient} construction. At this point we review filter quotients in a more detailed manner.

Given a category $\C$ with finite products and a suitable set of subterminal objects\footnote{An object $C$ is subterminal if the unique morphism to the terminal object is a mono.} $\Phi$, we can construct a new category $\C_\Phi$, known as the \emph{filter quotient category}, which has the same objects as $\C$, and where morphisms are given by equivalence classes of morphisms in $\C$, via an equivalence relation generated by $\Phi$. By construction it comes with a projection functor $P_\Phi\colon\C \to \C_\Phi$, which maps a morphism to its equivalence class. See \cref{subsec:filter quotient} for a detailed review.  

In this paper we show that if the original category is a model category $\cM$, then under suitable conditions on $\cM$ and $\Phi$, which do not include any of the ``small generation" assumptions, the filter quotient category $\cM_\Phi$ also admits a model structure (\cref{thm:filter quotient model structure}). Moreover, the projection functor $P_\Phi$ preserves many important model categorical properties. This includes weak equivalences, (co)fibrations, finite (co)limits (\cref{subsec:limits}), Cartesian closure (\cref{subsec:cartesian closure}), properness (\cref{subsec:properness}), and being simplicial or enriched (\cref{subsec:enriched model structures}). 

Having constructed filter quotient model categories, we also show compatibility by proving that the underlying $\infty$-category of the filter quotient model category recovers the already-mentioned \emph{filter quotient $\infty$-category} \cite{rasekh2021filterquotient} (\cref{sec:underlying infinity categories}). We culminate the paper with a careful analysis of examples, both classes of examples, namely \emph{filter products} (\cref{sec:filter products}), as well as specific examples (\cref{sec:examples}).

In light of the previous subsection, having constructed filter quotient model categories, we might wonder if they indeed provide us with new type-theoretical models. In two follow-up papers, we prove that the filter quotient construction can be used to construct new and non-trivial models of various homotopy type theories of interest \cite{rasekh2025filterhott,rasekh2025shott}. 

\subsection{Acknowledgment}
I would like to thank Gabin Kolly and Kathryn Hess for many helpful discussions in the context of the Bachelor project of Gabin Kolly at EPFL. I would also like to thank Emily Riehl for many useful discussions. Moreover, I am also grateful to the Max Planck Institute for Mathematics in Bonn for its hospitality and financial support. I am also grateful to the Hausdorff Research Institute for Mathematics in Bonn, Germany, for organizing the trimester ``Prospects of Formal Mathematics,'' funded by the Deutsche Forschungsgemeinschaft (DFG, German Research Foundation) under Germany's Excellence Strategy – EXC-2047/1 – 390685813, which resulted in many fruitful interactions and relevant conversations regarding foundation and type theory. Additionally, I am grateful to George Raptis for helpful comments and, in particular, suggesting \cite{raptisrosicky2018smallpresentations} as a relevant source. Finally, I would like to thank the anonymous referees for their careful reading and many helpful comments that improved the presentation and in particular the introduction of this work.

\section{Reviewing Model Categories and Filter Quotients}
This section serves as a review of relevant model categorical and filter quotient concepts. While the review of filter quotients follows standard references (\cref{subsec:filter quotient}), the review of model categories primarily follows the original source due to Quillen \cite{quillen1967modelcats}, and hence diverges from more modern approaches.

\subsection{Model Categories {\`a} la Quillen} \label{subsec:model categories}
Model category theory has been studied very extensively in the past decades, resulting in many detailed references \cite{quillen1967modelcats,hovey1999modelcategories,hirschhorn2003modelcategories,balchin2021modelcats}. However, various definitions of model categories presented in these citations have subtle differences, which will be of utmost importance to us. In this section we will hence explicate the relevant definitions that exhibit such nuances.

\begin{definition}
	Let $\cM$ be a category. A \emph{model structure} on $\cM$ is a triple $(\cF,\cC,\cW)$, called the \emph{fibrations}, \emph{cofibrations} and \emph{weak equivalences} (where $\cF \cap \cW$ are named trivial fibrations and $\cC \cap \cW$ are named trivial cofibrations), respectively, such that the following conditions hold:
	\begin{enumerate}[leftmargin=*]
		\item \textbf{Lifting:} Any commutative square with the left hand map a cofibration, and right hand map fibration, at least one is also a weak equivalence, admits a lift.
		\item \textbf{Factorization:} Every morphism admits two factorizations: a cofibration followed by a trivial fibration and a trivial cofibration followed by a fibration.
		\item \textbf{Retract:} The three classes are closed under retracts.
		\item \textbf{$2$-out-of-$3$:} For two composable morphisms $f,g$, if $2$ out of $f,g,g \circ f$ are in $\cW$ then so is the third.
	\end{enumerate}
\end{definition}

\begin{definition}
	A \emph{model category} is a category $\cM$ with finite limits and colimits and a model structure $(\cF,\cC,\cW)$.
\end{definition}

\begin{warning}
	This definition coincides with \cite{quillen1967modelcats}, but differs from \cite{hovey1999modelcategories,hirschhorn2003modelcategories,balchin2021modelcats} and many other modern sources, where model categories are assumed to be (co)complete.
\end{warning}

Despite the lack of (co)completeness, we can still recover many common model categorical definitions and constructions. For example, we do have an initial object $\emptyset$ and a terminal object $1$, which permits defining \emph{cofibrant} and \emph{fibrant} objects as those for which the unique morphism $\emptyset \to X$ is a cofibration, and $X \to 1$ is a fibration, respectively. Moreover, for an arbitrary object $X$, we can still construct a \emph{cofibrant replacement} $\emptyset \to \tilde{X} \to X$ as a trivial fibration out of a cofibrant object, and a \emph{fibrant replacement} $X \to \hat{X} \to 1$ as a trivial cofibration into a fibrant object, using the factorization axiom. 

Many other common model categorical constructions (properness, Cartesian closure) still coincide as well, and will be reviewed when appropriate. One major difference in the definition of model categories {\`a} la Quillen manifests in the context of enriched (and particularly simplicial) model categories, which we hence review here.

In this next definition $\V^{cmpt}$ denotes the full subcategory of compact objects in $\V$, where a \emph{compact object} is an object $K$ such that the functor $\Hom_\V(K,-)$ commutes with filtered colimits.
\begin{definition} \label{def:enriched	model structure}
	Let $(\V,I_\V,\otimes_\V)$ be a monoidal model category and $\cM$ a model category. A \emph{compactly enriched structure} on $\cM$ consists of the following data:
	\begin{enumerate}[leftmargin=*]
		\item $\cM$ is enriched over $\V$, meaning for two objects $X, Y$, we have an object in $\V$, denoted $\Map_\cM(X,Y)$, called the \emph{mapping object}, such that 
		\[\Hom_\V(I_\V,\Map_{\cM}(X,Y)) = \Hom_{\cM}(X,Y)\]
		along with unital and associative composition maps 
		\[\Map_{\cM}(X,Y) \otimes_\V \Map_{\cM}(Y,Z) \to \Map_{\cM}(X,Z) \] 
		in $\V$, for three objects $X,Y,Z$ in $\cM$.  
		\item Two functors:
		\begin{itemize}[leftmargin=*]
			\item \emph{Tensor:} $ - \otimes - \colon \cM \times \V^{cmpt} \to \cM$ 
			\item \emph{Cotensor:} $(-)^{(-)}\colon (\V^{cmpt})^{op} \times \cM \to \cM$ 
		\end{itemize}
		along with three isomorphisms, natural in $X,Y$ in $\cM$ and $K$	in $\V^{cmpt}$:
		\[\Hom_{\cM}(X \otimes K, Z) \cong \Hom_{\cM}(X, Y^K) \cong \Hom_{\V}(K, \Map_{\cM}(X,Y)) \] 
		\item For every cofibration $i\colon A \to B$ and fibration $p\colon Y \to X$, the induced map 
		\[(i^*,p_*)\colon\Map_{\cM}(B,Y) \to \Map_{\cM}(B,X) \times_{\Map_{\cM}(A,X)} \Map_{\cM}(A,Y)\] 
		is a fibration, and is trivial if either $i$ or $p$ is trivial.
	\end{enumerate}
\end{definition}

We can in particular consider this definition for $\V = \sset$ with the Kan model structure. 

\begin{definition} \label{def:simplicial model structure}
	A \emph{simplicial model category} is a compactly enriched model category $\cM$ over $\sset$ with the Kan model structure.
\end{definition}

More explicitly, the compact objects in $\sset$ are the \emph{finite simplicial sets},	meaning finite colimits of representables $\Delta^n$. Hence in the case of simplicial sets, simplicial model structures coincide with the original definition of Quillen \cite{quillen1967modelcats} and differ from more modern references, which would assume (co)tensor for all objects \cite{hovey1999modelcategories,hirschhorn2003modelcategories}. 

Finally, one minor difference in the definition of model categories {\`a} la Quillen manifests in the context of cofibrantly generated and combinatorial model categories, which are the final concepts we review here.

\begin{definition} \label{def:cofibrantly generated combinatorial}
	Let $\cM$ be a model category.
 \begin{itemize}[leftmargin=*]
	 \item $\cM$ is \emph{cofibrantly generated} if the underlying category of $\cM$ has small colimits, and is cofibrantly generated in the sense of \cite[Definition 11.1.2]{hirschhorn2003modelcategories}.
		\item $\cM$ is \emph{combinatorial} if the underlying category of $\cM$ is locally presentable, and $\cM$ is cofibrantly generated.
 \end{itemize}
\end{definition}

\begin{remark}
	Note the definition of cofibrantly generated model categories requires small objects \cite[Definition 10.4.1]{hirschhorn2003modelcategories} and transfinite compositions \cite[Definition 10.5.8]{hirschhorn2003modelcategories}, both of which require filtered colimits. However, the existence of filtered colimits and finite colimits implies the existence of all colimits. This means assuming the existence of small colimits in \cref{def:cofibrantly generated combinatorial} is the minimal possible condition required to characterize cofibrantly generated model categories. On the other hand, we do not need to make any additional assumptions regarding the existence of limits.
\end{remark}

\begin{remark}
	If the underlying category of $\cM$ is locally presentable, then it is in particular (co)complete. Hence being combinatorial in this case coincides with the modern definitions \cite{hovey1999modelcategories,hirschhorn2003modelcategories}.
\end{remark}

Having reviewed model categories, we also need to understand how to relate them. This is done via Quillen adjunctions.

\begin{definition} \label{def:quillen adjunction}
	Let $\cM$ and $\cN$ be two model categories. A \emph{Quillen adjunction} is an adjunction 
	\[  
	\begin{tikzcd}
	 \cM	\arrow[r, shift left=1.8, "F", "\bot"'] & \cN \arrow[l, shift left=1.8, "G"]
	\end{tikzcd}
	\] 
	such that $F$ preserves cofibrations and trivial cofibrations (equivalently, $G$ preserves fibrations and trivial fibrations).
\end{definition}

Quillen adjunctions provide us with a notion of equivalence between model categories. This requires the notion of derived (co)unit.

\begin{definition} 
	Let $(F \dashv G)$ be a Quillen adjunction between two model categories $\cM$ and $\cN$.
	\begin{itemize}[leftmargin=*]
		\item For a given object $X$ in $\cM$, the \emph{derived unit} is the map $X \xrightarrow{\eta_{X}} G F X \xrightarrow{G r_{F X}} G R F X$, where $r_{F X}\colon F X \to R F X$ is a fibrant replacement of $FX$ in $\cN$.
		\item For a given object $Y$ in $\cN$, the \emph{derived counit} is the map $F Q G Y \xrightarrow{F q_{G Y}} F G	Y \xrightarrow{\epsilon_{Y}} Y$, where $q_{G Y}\colon Q G Y \to G Y$ is a cofibrant replacement of $GY$ in $\cM$.
	\end{itemize} 
\end{definition}

\begin{definition} \label{def:quillen equivalence}
	Let $\cM$ and $\cN$ be two model categories and $(F \dashv G)$ a Quillen adjunction. 
	\begin{enumerate}[leftmargin=*]
		\item $(F \dashv G)$ is a \emph{homotopical localization} if	for every cofibrant object $X$ in $\cM$, the derived unit $X \to G R F X$ is a weak equivalence in $\cM$.
		\item $(F \dashv G)$ is a \emph{homotopical colocalization} if for every fibrant object $Y$ in $\cN$, the derived counit $F Q G Y \to Y$ is a weak equivalence in $\cN$.
		\item $(F \dashv G)$ is a \emph{Quillen	equivalence} if it is both a homotopical localization and a homotopical colocalization.
	\end{enumerate}
\end{definition}

\begin{remark}
 The term homotopical (co)localization used here matches the terminology due to Joyal and Tierney \cite[Definition 7.16, Proposition 7.17]{joyaltierney2007qcatvssegal}, but differs from other sources, such as \cite{cgmv2010localization,casacubertaraventostonks2021comparing}, where it corresponds to what is commonly known as a left or right Bousfield localization \cite{hirschhorn2003modelcategories}. See \cite[Example 5.1]{bataninwhite2024left} for an explicit example of a homotopical localization, in the sense of \cref{def:quillen equivalence}, that is provably not a left Bousfield localization.
\end{remark}

\subsection{Filter Quotients} \label{subsec:filter quotient}
Filter quotient categories have been studied extensively in the context of topos theory \cite{johnstone2002elephanti,maclanemoerdijk1994topos}. Here we review basic definitions and results.

\begin{definition}
	Let $(I,\leq)$ be a poset. A \emph{filter} $\Phi$ is a subset of $I$ that satisfies the following three properties:
	\begin{itemize}[leftmargin=*]
		\item \textbf{Non-emptiness:} $\Phi$ is non-empty.
		\item \textbf{Upward closure:} For every $x \in \Phi$ and $y \in I$, if $x \leq y$, then $y \in \Phi$.
		\item \textbf{Finite intersection:} For every $x,y \in \Phi$, there exists $z \in \Phi$ such that $z \leq x$ and $z \leq y$.
	\end{itemize}
\end{definition}

We now use filters to construct categories. Recall that in a category $\C$ with terminal object, an object $U$ is \emph{subterminal} if the unique morphism to the terminal object is a mono. Note this forms a poset with $U \leq V$ if there exists a (necessarily unique) morphism $U \to V$.

\begin{definition}
	Let $\C$ be a category with finite products. A \emph{filter of subterminal objects} on $\C$ is a filter on the poset of subterminal objects of $\C$. 
\end{definition}

Given a filter of subterminal objects, we can construct new categories. 

\begin{definition}
	Let $\C$ be a category with finite products and $\Phi$ be a filter of subterminal	objects. The \emph{filter quotient category}, denoted $\C_\Phi$, is a category with the following specifications:
	\begin{itemize}[leftmargin=*]
		\item \textbf{Objects:} The objects of $\C_\Phi$ are the objects of $\C$.
		\item \textbf{Morphisms:} The morphisms between two objects $X$ and $Y$ are given by equivalence classes of morphisms from $X \times U$ to $Y$, where $U$ is in $\Phi$. Here two morphisms $f\colon X \times U \to Y$ and $g\colon X \times V \to Y$ are equivalent if there exists $W \leq U, V$ in $\Phi$, such that the following diagram commutes: 
		\[ 
		\begin{tikzcd}
			X \times W \arrow[r, hookrightarrow] \arrow[d, hookrightarrow] & X \times U \arrow[d, "f"] \\
			X \times V \arrow[r, "g"] & Y
		\end{tikzcd}
		\]
	\end{itemize}
\end{definition}

The filter quotient category $\C_\Phi$ relates to the original category via the projection functor.

\begin{definition}
	Let $\C$ be a category with finite products and $\Phi$ be a filter of subterminal objects. The \emph{projection functor} $P_\Phi\colon \C \to \C_\Phi$ is the functor that is the identity on objects and takes a morphism to its equivalence class.
\end{definition}

Let us review some examples. Recall that a \emph{strict initial object} $\emptyset$ in a category $\C$ is an initial object such that every morphism with codomain $\emptyset$ is an isomorphism. This in particular implies that $\emptyset \times X \cong \emptyset$ for every object $X$ in $\C$. 

\begin{example} \label{ex:trivialquotient}
	Let $\C$ be a category with finite products and $\Phi$ be a filter consisting of only the terminal object $1$. Then $\C_\Phi$ is isomorphic to $\C$, as two morphisms $f,g\colon X \to Y$ are equivalent if and only if $f \circ \pi_1 = g \circ \pi_1$, which is equivalent to $f = g$.
\end{example}

\begin{example} \label{ex:principalfilter}
	Let $\C$ be a category with finite products, $U$ be a subterminal object, and $\Phi_U$ be the filter of all subterminal objects that include $U$, known as the \emph{principal filter generated by $U$}. Then, following the definition, $\C_{\Phi_U}$ can be described as a category with the same objects as $\C$ and morphisms $\Hom_{\C_U}(X,Y) = \Hom_{\C}(X \times U, Y)$. We can alternatively characterize $\C_{\Phi_U}$ as the slice	category $\C_{/U}$ of objects	over $U$.
\end{example}

\begin{example} \label{ex:initialquotient}
	Let $\C$ be	a category with finite products and strict initial object $\emptyset$. Let $\Phi$ be the filter consisting of all subterminal objects. Then any two morphisms in $\C_\Phi$ from $X$ to $Y$ are equal, as their restriction to $ X\times \emptyset = \emptyset$ are equal. Hence $\C_\Phi$ is the terminal category.
\end{example}

\begin{example} \label{ex:setsfilterquotient}
	Let $\C$ be a category with finite products such that the only subterminal objects are the terminal object $1$ and the strict initial object $\emptyset$. Then there are two possible filter quotient categories. If $\Phi	= \{1\}$, then by \cref{ex:trivialquotient} we have $\C_\Phi \cong \C$. If $\Phi = \{\emptyset, 1\}$, then by \cref{ex:initialquotient} we have that $\C_\Phi$ is the terminal category. This computation in particular applies to the categories $\set$, $\Top$, $\sset$.
\end{example}
	
\begin{example} \label{ex:pointed filter quotient}
 Let $\C$ be a pointed category	with finite products, meaning it has an object $0$ that is both terminal and initial. Then $\C$ has a single subterminal object, namely $0$. Indeed, for any subterminal object $U$, there are maps $0 \hookrightarrow U \hookrightarrow 0$, the composition of which is the identity, proving that $U \cong 0$. This means there is only one filter of subterminal objects, and the induced filter quotient is $\C$ itself. This applies in particular to the categories $\Grp$, $\Ab$, and any other category that naturally occurs in algebra.
\end{example}

We now have the following basic results about filter quotients. For a proof see \cite[Example D.5.1.7]{johnstone2002elephanti} or \cite[Section 9.4]{johnstone1977topos}.

\begin{theorem} \label{thm:filter quotient projection}
	Let $\C$ be a category with finite products and $\Phi$ be a filter of subterminal objects.
 \begin{enumerate}[leftmargin=*]
		\item $P_\Phi$ preserves finite (co)limits. Hence if $\C$ is finitely (co)complete, then so is $\C_\Phi$.
		\item $P_\Phi$ preserves exponential objects (Cartesian closure). Hence, if $\C$ is (locally) Cartesian closed, then so is $\C_\Phi$.
		\item $P_\Phi$ does not preserve infinite coproducts.
	\end{enumerate}
\end{theorem}

Finally, let us observe that filter quotients permit an alternative description as a filtered colimit, which will have useful formal implications later on. For a proof see \cite[Theorem V.9.2]{maclanemoerdijk1994topos}.

\begin{theorem} \label{thm:filtered colimits}
	Let $\C$ be a category with finite products and $\Phi$ be a filter of subterminal objects. Let $\C_{(-)}\colon\Phi^{op} \to \cat$ be the following functor:
	\begin{itemize}[leftmargin=*]
	\item It takes $U$ to the category $\C_U$ with the same objects as $\C$ and morphisms $\Hom_{\C_U}(X,Y) = \Hom_{\C}(X \times U, Y \times U)$.	
	\item It takes a morphism $U \to V$ (corresponding to the relation $U \leq V$ in the poset) to the functor $- \times U\colon \C_V \to \C_U$ (using $U \times V = U$).
	\end{itemize}
	Then $\C_\Phi$ is the filtered	colimit of the diagram $\C_{(-)}\colon\Phi^{op} \to \cat$.
\end{theorem}

\begin{lemma} \label{lem:induced functor}
	Let $\C,\D$ be categories with finite products and $\Phi_\C, \Phi_\D$ be filters of subterminal objects. Let $F\colon\C \to \D$ be a functor such that $F$ restricts to a functor $\Phi_\C \to \Phi_\D$ and $F(X \times U) = F(X) \times F(U)$, for all $X$ in $\C$ and $U$ in $\Phi$. Then $F$ induces a functor $F\colon \C_{\Phi_\C} \to \D_{\Phi_\D}$.  
\end{lemma}

\begin{proof}
	 By assumption, if $f, g\colon c \to d$ are equivalent, then $f \circ (\pi_W) = g \circ (\pi_W)$ for some $W \in \Phi_\C$, which implies that $F(f) \circ \pi_{F(W)} = F(g) \circ \pi_{F(W)}$, where  $F(W)$ in $\Phi_\D$, using the assumptions on $F$. Hence, $F$ induces a functor $\C_{\Phi_\C} \to \D_{\Phi_\D}$.
\end{proof}	

\begin{corollary} \label{cor:induced adjunction}
	Let $L\colon\C \to \D$ be a left adjoint with right adjoint $R\colon\D \to \C$. Moreover, assume $\C,\D$ have finite products and $L(X \times U) = L(X) \times L(U)$, for $X$ in $\C$ and $U$ in $\Phi$. Let $\Phi_\C$ ($\Phi_\D$) be a filter of subterminal objects on $\C$ ($\D$), such that $L$ ($R$) restricts to a functor $L\colon\Phi_\C \to \Phi_\D$ ($R\colon\Phi_\D \to \Phi_\C$). Then there is an induced adjunction between $\C_{\Phi_\C}$ and $\D_{\Phi_\D}$.
\end{corollary}

\begin{proof}
	By \cref{lem:induced functor}, there are induced functors $L\colon \C_{\Phi_\C} \to \D_{\Phi_\D}$ and $R\colon \D_{\Phi_\D} \to \C_{\Phi_\C}$. Moreover, the unit and counit map of the adjunction still satisfy the triangle identity in $\C_{\Phi_\C}, \D_{\Phi_\D}$, giving us the desired adjunction
\end{proof}

\section{Filter Quotient Model Categories}
We now prove that in suitable situations, filter quotients preserve model structures. 

\begin{definition}
	Let $\cM$ be a model category. An object $U$ is \emph{homotopically subterminal} if it is fibrant and the diagonal map $\Delta\colon U \to U \times U$ is a weak equivalence.
\end{definition}

\begin{example} \label{ex:terminal object}
 The terminal object $1$ is always homotopically	subterminal.
\end{example}

\begin{definition}
 Let $\cM$ be a model category. A homotopically subterminal object $U$ is called \emph{discrete} if it is also subterminal in $\cM$. 
\end{definition}

\begin{example} \label{ex:subterminal sset}
	Let $\sset^{Kan}$ be the Kan model structure on simplicial sets. Then a fibrant simplicial set $K$ (a Kan complex) is homotopically subterminal if it is empty or contractible. Among those the discrete ones are $\Delta^0$ and $\emptyset$. 
\end{example}

\begin{lemma} \label{lemma:simplicial subterminal}
	Let $\cM$ be a simplicial model category and $U$ be a fibrant object. Then the following are equivalent.
	\begin{enumerate}[leftmargin=*]
		\item $U$ is homotopically subterminal.
		\item For every cofibrant object $X$, the Kan complex $\Map(X,U)$ is homotopically subterminal.
	\end{enumerate}
\end{lemma}

\begin{proof}
	Direct implication of the fact that a morphism in $\cM$ is a weak equivalence if and only if for every cofibrant object $X$, the functor $\Map(X,-)$ takes it to a Kan equivalence.
\end{proof}

\begin{definition}
	Let $\cM$ be a model category. Define $\Subfib(\cM)$ as the poset with elements discrete homotopically subterminal objects and $U \leq V$ if $\Hom(U,V)$ is non-empty.  
\end{definition}

\begin{definition}
	Let $\cM$ be a model category. A \emph{filter of discrete homotopical subterminal objects} is a filter in the poset $\Subfib(\cM)$. 
\end{definition}

\begin{remark} \label{rem:terminal}
	The non-emptiness and \cref{ex:terminal object} implies that every filter of discrete homotopical subterminal objects includes the terminal object $1$.
\end{remark}

Given a model category $\cM$ and any filter of subterminal objects $\Phi$, we can construct a filter quotient $\cM_\Phi$. We now want to prove that for suitable $\Phi$, $\cM_\Phi$ comes with a model structure.

\begin{definition}
	Let $\cM$ be a model category and $\Phi$ be a filter of discrete homotopical subterminal objects. A class of morphisms $S$ is called \emph{$\Phi$-product stable}, if for every $f$ in $S$ and $U$ in $\Phi$, $f \times U$ is in $S$. 
\end{definition}

\begin{notation} \label{not:s phi}
	Let $\cM$ be a model category and $\Phi$ a filter of discrete homotopical subterminal objects. Let $S$ be a class of morphisms that is $\Phi$-product stable. Let $S_\Phi$ denote the class of morphisms $f$ in $\cM_\Phi$ with the property that there exists a $U$ in $\Phi$, such that $f \times U$ is in $S$.
\end{notation}

\begin{remark} \label{rem:s phi}
 Let us make some observations about \cref{not:s phi}.
	\begin{enumerate}[leftmargin=*]
		\item If $f$ is in $S_\Phi$ with the assumption that $f \times U$ is in $S$, then for all $W \leq U$, $f \times U \times W$ is also in $S$, as $S$ is $\Phi$-product stable. 
		\item The characterization is independent of a representative, as for $f,f' \in [f]$, $f \times U  = f' \times U$ for sufficiently large $U$ in $\Phi$.
		\item The map $P_\Phi\colon \cM \to	\cM_\Phi$ takes a morphism $f$ in $S$ to a morphism in $S_\Phi$, as $f \times 1$ is in $S$. 
	\end{enumerate} 
\end{remark}

Our aim is to apply these constructions to the three classes of morphisms defining model categories. However, this will not necessarily work as we would like. We do have the following immediate result regarding fibrations.

\begin{lemma}
 Let $\cM$ be a model category and $\Phi$ a filter of discrete homotopical subterminal objects. The class of fibrations is $\Phi$-product stable.
\end{lemma}

This does not generalize to other classes.

\begin{example}
 Let $\Fun([1] \times [1], \sset^{Kan})^{proj}$ be the projective model structure on the Kan model structure on $\sset$. Then an object $F\colon [1] \times [1] \to \sset$ is cofibrant if and only if the induced map $F(1,0) \coprod_{F(0,0)} F(0,1) \to F(1,1)$ is a cofibration.
	
	Let $\Phi = \{U_1,U_2,U_3\}$ be the following (principal) filter (\cref{ex:principalfilter}).
\[ \left\{
	\begin{tikzcd}[row sep=0.15cm, column	sep=0.6cm]
		\emptyset \arrow[r] \arrow[dd] & \Delta^0 \arrow[dd] &[-.6cm] &[-.6cm] \emptyset \arrow[r] \arrow[dd] & \Delta^0 \arrow[dd] &[-.6cm] &[-.6cm] \Delta^0 \arrow[r] \arrow[dd] & \Delta^0 \arrow[dd] &[-0.6cm] \\ 
		&  & , & & & , & & & \\
		\emptyset \arrow[r] & \Delta^0  & & \Delta^0 \arrow[r] & \Delta^0 & & \Delta^0 \arrow[r] & \Delta^0 
	\end{tikzcd}
	\right\}.
\]
In the projective model structure the fibrations are point-wise, hence $\Phi$ is indeed a filter of discrete homotopical subterminal objects. Let $\sigma\colon \Delta^0 \coprod \Delta^0 \to \Delta^0 \coprod \Delta^0$ be the unique non-trivial automorphism. Now, the object $F$, defined as the left hand square  
\[
 \begin{tikzcd}[row sep=0.15cm, column	sep=0.6cm]
		 &[-.8cm] \Delta^0 \coprod \Delta^0 \coprod \Delta^0 \coprod \Delta^0 \arrow[r,  "\sigma + \id"] \arrow[dd, "\id + \id"] & \Delta^0 \coprod \Delta^0 \arrow[dd] &[-1cm]  &[1cm] &[-.8cm] \emptyset \arrow[r] \arrow[dd] & \Delta^0 \coprod \Delta^0 \arrow[dd] \\ 
		 F = & &  &\strut \arrow[r, Rightarrow] & F \times U_1 =  & & &\\
		 & \Delta^0 \coprod \Delta^0 \arrow[r] & \Delta^0  & & & \emptyset \arrow[r] & \Delta^0  
	\end{tikzcd} 
\]
is cofibrant, as it is a pushout square, but $F \times U_1$, given on the right, is not, as the induced map $\Delta^0 \coprod \Delta^0 \to \Delta^0$ is not a cofibration in the Kan model structure. Hence the class of cofibrations is not $\Phi$-product stable.
\end{example}

\begin{example}
	Let $\sset^{Joy}$ be the Joyal model structure on the category of simplicial sets, which induces a model structure on the slice-category $(\sset_{/\Delta^2})^{Joy}$. Let $<02>\colon \Delta^1 \to \Delta^2$ be the unique map given by $0 \mapsto 0$, $1 \mapsto 2$, which is indeed a fibration, and let $\Phi$ be the principal filter of subobjects in $(\sset_{/\Delta^2})^{Joy}$ consisting of all subobjects $\Delta^2$ that include $<02>$ (\cref{ex:principalfilter}). Then the map $\Lambda^2_1 \to \Delta^2$ is a weak equivalence, but 
	\[ \Delta^0 \coprod \Delta^0 = \Lambda^2_1 \times_{\Delta^2} \Delta^1 \to \Delta^2 \times_{\Delta^2} \Delta^1  = \Delta^1 \]
	is not a weak equivalence in $(\sset_{/\Delta^2})^{Joy}$, and so the	class of weak equivalences is not $\Phi$-product stable.
\end{example}

\begin{definition}
	Let $\cM$ be a model category. A \emph{model filter} is a filter of discrete homotopically subterminal objects $\Phi$ such that cofibrations and weak equivalences are $\Phi$-product stable. 
\end{definition}

\begin{theorem} \label{thm:filter quotient model structure}
	Let $\cM$ be a model category and denote the fibrations/cofibrations/weak equivalences by $\cF/\cC/\cW$. Let $\Phi$ be a model filter on $\cM$. Then $\cM_\Phi$ forms a model structure with fibrations/cofibrations/weak equivalences given by $\cF_\Phi/\cC_\Phi/\cW_\Phi$.
\end{theorem}

\begin{proof}
	We separately confirm the axioms of a model category. 
 \begin{itemize}[leftmargin=*, itemsep=.2cm]
		\item[] 
	\textbf{Finite (co)limits:} Finite (co)completeness of $\cM_\Phi$ follows from finite (co)completeness of $\cM$ and \cref{thm:filter quotient projection}.
	
	\item[] 
 \textbf{Lifting:} Assume we have the following diagram in $\cM_\Phi$:
	\[
		\begin{tikzcd}
			A \arrow[r, "f"] \arrow[d, hookrightarrow, "i"'] & Y \arrow[d, "p", twoheadrightarrow] \\ 
			B \arrow[r, "g"'] & X
		\end{tikzcd},
	\]
	where $p$ is a fibration and $i$ is a cofibration. Without loss of generality we will also assume that $i$ is a weak equivalence, as the other case is analogous. Pick $W_1$ in $\Phi$ such that $(pf) \times W_1 = (gi) \times W_1$, pick $W_2$ in $\Phi$ such that $i \times W_2$ is a cofibration, pick $W_3$ in $\Phi$ such that $i \times W_3$ is a weak equivalence, and pick $W_4$ in $\Phi$ such that $p \times W_4$ is a fibration. Pick $W \leq W_1, W_2, W_3, W_4$ in $\Phi$. We now have the following diagram in $\cM$
	\[
		\begin{tikzcd}[row	sep=1cm, column sep=1cm]
			W \times A \arrow[r, "W \times f"] \arrow[d, hookrightarrow, "W \times i"'] & W \times Y \arrow[d, "W \times p", twoheadrightarrow] \\ 
			W \times B \arrow[r, "W \times g"'] \arrow[ur, dashed, "h" description] & W \times X
		\end{tikzcd}.
	\]
	By construction this diagram commutes in $\cM$, $W \times i$ is a trivial cofibration and $W \times p$ is a fibration and so has a lift $h$. Now the class of $h$ in $\cM_\Phi$ is the desired lift of the original diagram.
 
	\item[] 
	\textbf{Factorization:} Let $f\colon X \times U \to Y$ be a representative of an arbitrary morphism in $\cM_\Phi$. By the factorization axiom in $\cM$, we can factor $f$ as $f = p \circ i$, where $i$ is a trivial cofibration and $p$ is a fibration. By \cref{rem:s phi} (3), the equivalence class of $i$ ($p$) is a trivial cofibration (fibration) in $\cM_\Phi$ and we still have $f = p \circ i$ in $\cM_\Phi$, giving us the desired factorization. 

	\item[]
	\textbf{Retracts:} Assume that $f$ is a fibration in $\cM_\Phi$ and $g$ is a retract of $f$. Pick $W_1$ in $\Phi$, so that $g \times W_1$ is a retract of $f \times W_1$ in $\cM$ and a $W_2$, such that $f \times W_2$ is a fibration in $\cM$. Let $W \leq W_1, W_2$ in $\Phi$, then $g \times W$ is a retract of the fibration $f \times W$ and so a fibration as well, which implies that $g$ is a fibration in $\cM_\Phi$. The case for cofibrations and weak equivalences is analogous.

	\item[] 
	\textbf{$2$-out-of-$3$:} Let $f,g$ be two composable morphisms in $\cM_\Phi$. We prove that if $f$ and $g$ are weak equivalences, then so is $gf$, the other two cases are analogous. Pick $W_1$ in $\Phi$, such that $f \times W_1$ is a weak equivalence and $W_2$, such that $g \times W_2$ is a weak equivalence in $\cM$. Now let $W \leq W_1, W_2$ in $\Phi$. Then $gf \times W$ is a weak equivalence in $\cM$, and so $gf$ is a weak equivalence in $\cM_\Phi$.   
\end{itemize}
\end{proof}
 
\begin{remark}
	The functor $P_\Phi\colon \cM \to \cM_\Phi$ preserves (co)fibrations and weak equivalences, however is not left or right Quillen, as it is not a left or right adjoint.
\end{remark}

\section{From the original to the filter quotient}
In this section we aim to discuss which classical model categorical properties are preserved via the filter quotient construction. For every case that fails we present a counterexample.

\subsection{Limits and Colimits} \label{subsec:limits}
We have already established that $P_\Phi$ preserves finite (co)limit diagrams and hence preserves finite (co)completeness (\cref{thm:filter quotient projection}). However, this does not generalize to infinite (co)completeness. For an explicit counterexample see \cref{ex:filter product sset}. 

\subsection{Cartesian Closure} \label{subsec:cartesian closure}
We will prove $P_\Phi$ preserves internal mapping objects and, hence, $P_\Phi$ preserves the property of being Cartesian closed.

\begin{notation}
	Let $\cM$ be a model category. For two morphisms $i\colon A \to B, j\colon C \to D$ in $\cM$, we denote the pushout product of $i$ and $j$, $A \times D \coprod_{A \times C} B \times C \to B \times D$ by $i \square j$. 
\end{notation}

Recall that a model category $\cM$ is \emph{Cartesian closed} if its underlying category is Cartesian closed and for two cofibrations $i\colon A \to B$ and $j\colon C \to D$, the pushout product $i \square j$ is a cofibration that is a trivial cofibration if one of $i$ or $j$ is a trivial cofibration.

\begin{proposition} \label{prop:cartesian closure}
 Let $\cM$  be a model category and $\Phi$ a model filter. Then $P_\Phi$ preserves pushout products and pullback-exponentials. Hence, if $\cM$ is a Cartesian closed model category, then so is $\cM_\Phi$.
\end{proposition}

\begin{proof}
	Let $\cM$ be a Cartesian closed model category. By \cref{thm:filter quotient projection}, the underlying category of $\cM_\Phi$ is Cartesian closed and $P_\Phi$ preserves Cartesian closure. Let $i\colon A \to B$ and $j\colon C \to D$ be two cofibrations in $\cM_\Phi$. We need to prove that the pushout product $i \square j$ is a cofibration and is a trivial cofibration if one of $i$ or $j$ are a trivial cofibration. 

	Pick $W_1$ and $W_2$ in $\Phi$, such that $i \times W_1$ and $j \times W_2$ are cofibrations in $\cM$ and let $W \leq W_1, W_2$. By direct computation $(i \square j) \times W \cong (i \times W) \square (j \times W)$, using the fact that $- \times W$ commutes with products and pushouts. As $\cM$ is Cartesian closed, $(i \times W) \square (j \times W)$ is also a cofibration, and so $i \square j$ is a cofibration in $\cM_\Phi$. The case where $i$ or $j$ is a trivial cofibration follows analogously. Hence, $\cM_\Phi$ is a Cartesian closed model category.
\end{proof}

\subsection{Local Presentability and Cofibrant Generation}
There are examples where $\cM$ is cofibrantly generated or even combinatorial, but $\cM_\Phi$ is not. See \cref{ex:filter product sset,ex:filter product top} for explicit counterexamples. 

\subsection{Properness} \label{subsec:properness}
Let us prove that $\cM_\Phi$ is (right or left) proper if $\cM$ is (right or left) proper. Recall that a model category is \emph{right proper} (\emph{left proper}) if the pullback of a weak equivalence along a fibration (pushout along a cofibration) is again a weak equivalence.

\begin{proposition} \label{prop:right proper}
	Let $\cM$ be a model category and $\Phi$ a model filter. If $\cM$ is right proper, then $\cM_\Phi$ is right proper.
\end{proposition}

\begin{proof}
 Let $p\colon Y \to X$ be a fibration and $f\colon Z \to X$ be a weak equivalence in $\cM_\Phi$. Pick $W_1$ in $\Phi$ such that $p \times W_1$ is a fibration in $\cM$ and $W_2$ in $\Phi$, such that $f \times W_2$ is a weak equivalence in $\cM$, and let $W \leq W_1, W_2$ in $\Phi$. As $\cM$ is right proper, the pullback $(p \times W)^*(f\times W) \cong p^*f \times W$ is a weak equivalence in $\cM$, meaning $p^*f$ is a weak equivalence in $\cM_\Phi$. Finally, as $P_\Phi$ preserves pullbacks, $p^*f \times W$ is also the pullback in $\cM_\Phi$, hence we are done. 
\end{proof}

\begin{proposition} \label{prop:left proper}
		Let $\cM$ be a model category and $\Phi$ a model filter. Assume that for every $U$ in $\Phi$, $U \times -$ preserves pushouts. If $\cM$ is left proper, then $\cM_\Phi$ is left proper.
\end{proposition}

\begin{proof}
	We follow the same steps as in \cref{prop:right proper}, using the fact that $P_\Phi$ and $- \times U$ preserve pushouts.
\end{proof}

We now have the following immediate corollary.

\begin{corollary} \label{cor:proper}
	Let $\cM$ be a model category and $\Phi$ a model filter. Assume that for every $U$ in $\Phi$, $U \times -$ preserves pushouts. If $\cM$ is proper, then $\cM_\Phi$ is proper.
\end{corollary}

\subsection{Enriched Model Structures} \label{subsec:enriched model structures}
We now want to study the interaction between filter quotients and enriched model structures. Let $\cM$ be a compactly enriched model category and $\Phi$ a model filter. We want to construct a compactly enriched model structure on the model category $\cM_\Phi$. Constructing the enriched mapping object on $\cM_\Phi$ is straightforward. 

\begin{definition} \label{def:mapphi}
	Let $(\V,\otimes_\V,I_\V)$ be a monoidal model category with filtered colimits. Let $\C$ be a $\V$-enriched category and $\Phi$ a filter of subterminal objects. We define the \emph{induced enriched mapping object}, 
	\[\Map_{\C_\Phi}(-,-)\colon \C_\Phi^{op} \times \C_\Phi \to \V, \]
	by $\Map_{\C_\Phi}(X,Y) = \colim_{U \in \Phi} \Map_{\C}(X \times U, Y \times U)$.
\end{definition}

\begin{lemma} \label{lemma:mapphi}
		Let $(\V,\otimes_\V,I_\V)$ be a monoidal model category with filtered colimits, such that the functor $ - \otimes_\V -$ commutes with filtered colimits. Let $\C$ be a $\V$-enriched category and $\Phi$ a filter of subterminal objects. The induced enriched mapping object $\Map_{\C_\Phi}(X,Y)$ makes $\C_\Phi$ a $\V$-enriched category.
\end{lemma}

\begin{proof}
	Let us denote by $\comp_{X,Y,Z}\colon	\Map_{\C}(X,Y) \times \Map_{\C}(Y,Z) \to \Map_{\C}(X,Z)$ the composition map of $\C$. Now, for three objects $X, Y, Z$ in $\C_\Phi$, we have the following map 
	\begin{align*}
		\Map_{\C_\Phi}(X,Y) \otimes_\V \Map_{\C_\Phi}(Y,Z) & =  \underset{U \in \Phi}{\colim} \ \Map_{\C}(X \times U,Y \times U) \otimes_\V \underset{U \in \Phi}{\colim} \ \Map_{\C}(Y \times U,Z \times U) \\
		& \cong   \underset{U \in \Phi}{\colim} \ \Map_{\C}(X \times U,Y \times U) \otimes_\V \Map_{\C}(Y \times U,Z \times U) \\
		& \xrightarrow{\underset{U \in \Phi}{\colim} \ \comp_{X\times U, Y \times U, Z \times U}}  \underset{U \in \Phi}{\colim} \ \Map_{\C}(X \times U,Z \times U) \\ 
		& =  \Map_{\C_\Phi}(X,Z).
	\end{align*}
	Here the isomorphism follows from the fact that filtered colimits commute with the monoidal structure. As each $\comp_{X\times U, Y \times U, Z \times U}$ is associative and unital, it follows that this composition map is	associative and unital, as well. 
\end{proof}

Of course we need to make sure that this enriched mapping object is compatible with the hom set in the filter quotient, which is the aim of the next lemma, the proof of which directly follows from the fact that compact objects commute with filtered colimits, and our description of hom sets in filter quotient categories (\cref{thm:filtered colimits}).

\begin{lemma} \label{lemma:compact unit}
	Let $(\V,\otimes_\V,I_\V)$ be a monoidal model category with filtered colimits, such that the unit $I_\V$ is compact. Then \[\Hom(I_\V,\Map_{\cM_\Phi}(X,Y)) = \Hom_{\cM_\Phi}(X,Y).\] 
\end{lemma}

We now want to construct an enriched tensor and cotensor on $\cM_\Phi$. This requires additional conditions on the enrichment category $\V$ and the filter $\Phi$. 

\begin{definition} \label{def:quotient compatible}
	Let $(\V,\otimes_\V,I_\V)$ be a monoidal model category. We say $\V$ is \emph{quotient compatible} if it satisfies the following conditions:
	\begin{enumerate}[leftmargin=*]
		\item $\V$ has filtered colimits. 
		\item $- \otimes_\V -$	commutes with filtered colimits.
		\item The unit $I_\V$ is compact.
		\item (Trivial) fibrations are closed under filtered	colimits.
	\end{enumerate}
\end{definition}

The last condition might appear restrictive, however, it is satisfied in many relevant cases, as the following lemma demonstrates, which is proven in \cite[Lemma 7.4.1]{hovey1999modelcategories}.

\begin{lemma}\label{lemma:cofibrant generation}
	Let $\V$ be a monoidal model category with filtered colimits. If there is a set of generating (trivial) cofibrations with (co)domain compact objects, then (trivial) fibrations are closed under filtered	colimits. Hence, if $\V$ also has a compact unit and $- \otimes_\V -$	commutes with filtered colimits, then $\V$ is quotient compatible. 
\end{lemma}
	
\begin{definition}\label{def:simplicial model filter}
	Let $(\V,\otimes_\V,I_\V)$ be a monoidal model category. Let $\cM$ be a $\V$-compactly enriched model category. A \emph{$\V$-enriched model filter} $\Phi$ is a model filter on $\cM$ such that for all objects $X$ in $\cM$, compact object $K$ in $\V$ and $U$ in $\Phi$, $(X \otimes K) \times U = (X \times U) \otimes K$.
\end{definition}

As we will see further below (for example \cref{prop:filter product enriched} and more explicitly \cref{ex:filter product sset}), the conditions of \cref{def:simplicial model filter} are in fact satisfied in many relevant cases.

\begin{remark} \label{rem:simplicial model filter}
	Note for a model filter to be $\V$-enriched depends solely on the underlying category of $\V$, and is independent of the model structure on $\V$.
\end{remark}

\begin{lemma}\label{lemma:induced co tensor}
	Let $(\V,\otimes_\V,I_\V)$ be a quotient compatible monoidal model category. Let $\cM$ be a $\V$-compactly enriched model category and $\Phi$ a  $\V$-enriched model filter. Then the $\V$-enriched category $\cM_\Phi$, with enrichment $\Map_\Phi(-,-)$ from \cref{def:mapphi}, is tensored and cotensored over $\V^{cmpt}$, and $P_\Phi\colon \cM \to \cM_\Phi$ preserves (co)tensors. 
\end{lemma}

\begin{proof}
	Let $K$ be a compact object in $\V$, and $X, Y$ objects in $\cM$. We prove $X \otimes K$ and $Y^K$ satisfy the universal properties of the (co)tensor, which also directly prove $P_\Phi$ preserves (co)tensors. We have the following two chains of natural bijections:
	\begin{align*}
		\Hom_{\cM_\Phi}(X \otimes K, Y) & \cong \colim_{U \in \Phi}\Hom_{\cM}((X \otimes K) \times U, Y) & \text{\cref{thm:filtered colimits}}\\
		& \cong \colim_{U \in \Phi}\Hom_{\cM}((X \times U) \otimes K, Y)  & \text{\cref{def:simplicial model filter}} \\ 
		& \cong \colim_{U \in \Phi}\Hom_{\cM}((X \times U), Y^K) & \text{cotensor on } \cM\\ 
		& \cong \Hom_{\cM_\Phi}(X, Y^K) & \text{\cref{thm:filtered colimits}} 
	\end{align*}	

	and 
	\begin{align*}
		\Hom_{\cM_\Phi}(X \otimes K, Y) & \cong \colim_{U \in \Phi}\Hom_{\cM}((X \otimes K) \times U, Y) & \text{\cref{thm:filtered colimits}} \\ 
		& \cong \colim_{U \in \Phi}\Hom_{\cM}((X \times U) \otimes K, Y) & \text{\cref{def:simplicial model filter}} \\
		& \cong \colim_{U \in \Phi}\Hom_{\V}(K, \Map_{\cM}(X \times U,Y)) & \text{enriched mapping on } \cM\\
		& \cong \Hom_{\V}(K, \colim_{U \in \Phi} \Map_{\cM}(X \times U,Y)) & \text{K is compact}\\
		& \cong \Hom_{\V}(K, \Map_{\cM_\Phi}(X \times U,Y)) & \text{\cref{def:mapphi}}
	\end{align*}
\end{proof}

\begin{remark}
	Notice from the proof, the natural isomorphism between tensor and cotensor holds in fact for all objects in $\V$.
\end{remark}

We can now prove that the filter quotient construction does preserve compactly enriched model structures.

\begin{theorem} \label{thm:enriched filter quotient}
	Let $(\V,\otimes_\V,I_\V)$ be a quotient compatible monoidal model category. Let $\cM$ be a $\V$-compactly enriched model category and $\Phi$ a $\V$-enriched model filter. Then $\cM_\Phi$ is a $\V$-compactly enriched model category. Moreover, $P_\Phi$ preserves (co)tensors.
\end{theorem}

\begin{proof}
	We have already shown that $\cM_\Phi$ is enriched and (co)tensored, and the (co)tensor is preserved by $P_\Phi$ (\cref{lemma:induced co tensor}). We also proved that the mapping object is compatible with the underlying hom set (\cref{lemma:compact unit}). So, we only need to verify the third axiom of \cref{def:enriched model structure}. Let $i\colon A \to B$ be a cofibration and $p\colon Y \to X$ be a fibration in $\cM_\Phi$. Choose a $U$ in $\Phi$, such that $i \times U$ is a cofibration in $\cM$ and $p \times U$ is a fibration in $\cM$. By assumption, the induced map $((i \times U)^*, (p \times U)_*)$ is a fibration, which is trivial if either $i$ or $p$ is trivial. 
	
	As filtered colimits commute with finite limits, $([i]^*,[p]_*) = \colim_{V \in \Phi}((i \times V)^*, (p \times V)_*)$. Now the inclusion from $\Phi_U$, the elements in $\Phi$ larger than $U$, into $\Phi$ is final, meaning this colimit coincides with the colimit  $\colim_{V \in \Phi_U}((i \times V)^*, (p \times V)_*)$. Finally, by assumption on $U$, every morphism in this colimit is a (trivial) fibration, which are stable under filtered colimits, by \cref{def:quotient compatible}. Hence, we are done.
\end{proof}

Let us consider some cases of interest. First, we observe that for many cofibrantly generated model categories, the conditions of \cref{def:quotient compatible} are straightforward.

\begin{corollary} \label{cor:cofibrantly generated}
	Let $\V$ be a Cartesian monoidal cofibrantly generated model structure, with (co)domains of generating (trivial) cofibrations compact and compact unit (meaning a compact terminal object). Then every enriched model filter $\Phi$ on a $\V$-enriched model category $\cM$ induces a $\V$-compactly enriched model structure $\cM_\Phi$. 
\end{corollary}

\begin{proof}
 By \cref{lemma:cofibrant generation} and the fact that binary products commute with filtered colimits,  $\V$ is in fact quotient compatible, so the result follows from \cref{thm:enriched filter quotient}.
\end{proof}

We can also reduce the condition on enriched filters.

\begin{lemma} \label{lemma:enriched filter presheaf}
	Let $\V = \Fun(\C^{op},\set)$ be a presheaf category, considered as a Cartesian monoidal category. Let $\cM$ be a $\V$-enriched model category and $\Phi$ a model filter, such that $- \times U\colon \cM \to \cM$ preserves finite colimits. Then the following are equivalent:
	\begin{enumerate}[leftmargin=*]
		\item $\Phi$ is an enriched model filter.
  \item For every $c$ in $\C$, $(X \otimes \Hom_\C(-,c)) \times U = (X \times U) \otimes \Hom_\C(-,c)$.
	\end{enumerate}
\end{lemma}

\begin{proof}
	By assumption, the two functors $(X \otimes -) \times U$, $(X \times U) \otimes -$ preserve finite colimits. The equivalence now follows from the fact that a functor from $\V^{cmpt}$ to $\cM$ that preserves finite colimits is uniquely determined by its values on the representables $\Hom_\C(-,c)$. 
\end{proof}

We can now apply the general result to two examples of interest. First of all, following common convention, if a model filter is enriched over any model structure on $\sset$, such as the Kan or Joyal model structure (\cref{rem:simplicial model filter}), then we	call it a \emph{simplicial model filter}.

\begin{corollary} \label{cor:simplicial	model category}
	Let $\cM$ be a simplicial model category and $\Phi$ a simplicial model filter. Then $\cM_\Phi$ is a simplicial model category.
\end{corollary}

\begin{proof}
	 The Kan model structure on $\sset$ is cofibrantly generated with (co)domains of generating (trivial) cofibrations compact (horn and boundary inclusions) and compact terminal object $(\Delta^0)$.	Hence, by \cref{cor:cofibrantly generated}, $\cM_\Phi$	is a simplicial model category. 
\end{proof}

Let us consider a second example of interest, the proof of which follows similarly to \cref{cor:simplicial model category}

\begin{corollary} \label{cor:model category enriched joyal}
	Let $\cM$ be a model category enriched over the Joyal model structure and $\Phi$ a simplicial model filter. Then the model category $\cM_\Phi$ is enriched over the Joyal model structure.
\end{corollary}

\subsection{Quillen Adjunctions and Equivalences} \label{subsec:quillen adjunctions}
Finally, we observe that filter quotients preserve Quillen adjunctions and Quillen equivalences. Here we rely on \cref{def:quillen adjunction,def:quillen equivalence}.

\begin{proposition} \label{prop:quillen adjunction}
	Let $\cM, \cN$ be model categories. Let $\Phi_\cM$ be a model filter on $\cM$ and $\Phi_\cN$ a model filter on $\cN$. Let 
	\[
		\begin{tikzcd}[column sep = 1cm]
			\cM \arrow[r, "F", shift left = 1.8] \arrow[r, "\bot", leftarrow, shift right = 1.8, "G"'] & \cN 
		\end{tikzcd}
	\]
	be a Quillen adjunction, such that for all $X$ in $\cM$ and $U$ in $\Phi_\cM$, $F(X \times U) = F(X) \times F(U)$, $F(U)$ in $\Phi_{\cN}$, and for all $V$ in $\Phi_{\cN}$, $G(V)$ in $\Phi_{\cM}$. Then, this induces a Quillen adjunction  
	\[
		\begin{tikzcd}[column sep = 1cm]
			\cM_{\Phi_\cM} \arrow[r, "F", shift left = 1.8] \arrow[r, "\bot", leftarrow, shift right = 1.8, "G"'] & \cN_{\Phi_\cN} 
		\end{tikzcd}.
	\]
	Moreover, if the original adjunction $(F,G)$ is a homotopical (co)localization or Quillen equivalence, then so is the induced adjunction. 
\end{proposition}

\begin{proof}
	The existence of the adjunction is a direct implication of \cref{cor:induced adjunction}. We show that $F_\Phi$ is left Quillen. Let $i\colon	A \to B$ be a (trivial) cofibration in $\cM_\Phi$. This means there exists $U$ in $\Phi_{\cM}$, such that $i \times U$ is a (trivial) cofibration in $\cM$. By assumption $F(i \times U) = F(i) \times U$ is a (trivial) cofibration, which implies that the class $F(i)$ is a (trivial) cofibration in $\cN$. Hence, $F_\Phi$ is left Quillen. 

	Now, if the original adjunction is a homotopical (co)localization or Quillen equivalence, then one of the derived unit or derived counit maps (or both) are equivalences, which remain equivalences in $\cM_{\Phi_\cM}, \cN_{\Phi_\cN}$, as the projection functors preserve equivalences.
\end{proof}

\begin{proposition} \label{prop:quillen adjunction enriched}
	Let $(\V,I_\V,\otimes_\V)$ be a quotient compatible monoidal model category, $\cM, \cN$ two $\V$-compactly enriched model categories, and $\Phi_\cM$ ($\Phi_\cN$) a $\V$-enriched model filter on $\cM$ ($\cN$). Let 
	\[
		\begin{tikzcd}
			\cM \arrow[r, "F", shift left = 1.8] \arrow[r, "\bot", leftarrow, shift right = 1.8, "G"'] & \cN 
		\end{tikzcd}
	\]
	be a Quillen adjunction, such that the induced adjunction 
	\[
		\begin{tikzcd}
			\cM_{\Phi_\cM} \arrow[r, "F", shift left = 1.8] \arrow[r, "\bot", leftarrow, shift right = 1.8, "G"'] & \cN_{\Phi_\cN} 
		\end{tikzcd}
	\]
 exists. Then the induced adjunction is an enriched Quillen adjunction. 
	\end{proposition}

	\begin{proof}
		We only need to observe that $F$ commutes with tensors. However, by \cref{thm:enriched filter quotient}, $P_\Phi$ preserves tensors, so the result follows directly. 
	\end{proof}

\section{From Model Categories to \texorpdfstring{$\infty$}{oo}-Categories} \label{sec:underlying infinity categories}
In this section we confirm that the construction coincides with the filter quotient construction developed in \cite{rasekh2021filterquotient}. Let us quickly review the main result. 

\begin{definition}
	An \emph{$\infty$-category} $\M$ is a category strictly enriched over the category of Kan complexes.
\end{definition}

\begin{definition}
	Let $\M$ be an $\infty$-category. A subterminal object is an object $X$, such that for all objects $Y$, the Kan complex $\Map_{\M}(X,Y)$ is homotopically subterminal (\cref{ex:subterminal sset}).
\end{definition}

We now have the following main result.

\begin{theorem}[{\cite[Proposition 2.2]{rasekh2021filterquotient}}] \label{thm:filter quotient infinity category}
	Let $\M$ be an $\infty$-category and $\Phi$ a filter of subterminal objects in $\M$. Then there exists an $\infty$-category $\M_\Phi$ with the same objects and with morphisms given by the filtered colimit 
	\[\Map_{\M_\Phi}(X,Y) = \colim_{U \in \Phi}\Map_{\M}(X \times U, Y \times U)\] 
\end{theorem}

Note, part of this result includes the fact that the ordinary colimit of Kan complexes is again a Kan complex.

\begin{definition}
	Let $\cM$ be a simplicial model category. The \emph{underlying $\infty$-category of $\cM$}, denoted $\Ho_\infty(\cM)$, is defined as the Kan enriched category of fibrant cofibrant objects in $\cM$.  
\end{definition}

\begin{lemma}
	Let $\cM$ be a simplicial model category and $\Phi$ a model filter, such that for all $U$ in $\Phi$, $U$ is cofibrant. Then $\Phi$ is a poset of subterminal objects in the underlying $\infty$-category $\Ho_\infty(\cM)$.
\end{lemma}

\begin{proof}
	Let $U$ be an object in $\Phi$. By assumption, $U$ is cofibrant and fibrant, meaning it is an object in $\Ho_\infty(\cM)$. Now, by \cref{lemma:simplicial subterminal}, $\Map_{\cM}(-, U)$ is a $(-1)$-truncated Kan complex, which means that $U$ is subterminal in $\Ho_\infty(\cM)$. 
\end{proof}

\begin{theorem}
 Let $\cM$ be a simplicial model category and $\Phi$ be a model filter with all elements cofibrant. Then the underlying $\infty$-category of the filter quotient model category is equivalent to the filter quotient $\infty$-category. 
\end{theorem}

\begin{proof}
	By definition, $\Ho_\infty$ and $(-)_\Phi$ preserve the set of objects. This means the two Kan enriched categories $\Ho_\infty(\cM_\Phi)$ and $\Ho_\infty(\cM)_\Phi$ have the same set of objects. Moreover, by \cref{def:mapphi,cor:simplicial	model category,thm:filter quotient infinity category} they have the same mapping Kan complexes.  
\end{proof}

The comparison functor provides us with an effective way to compute homotopy categories of filter quotient model structures. Recall that for a simplicial model category $\cM$, the homotopy category $\Ho(\cM)$ is defined as the category with objects the fibrant-cofibrant objects and morphisms $\Hom_{\Ho(\cM)}(X,Y) = \pi_0(\Map_{\cM}(X,Y))$, for objects $X,Y$ in $\Ho(\cM)$. Homotopy categories of $\infty$-categories are defined analogously. We now have the following result.

\begin{proposition} \label{prop:homotopy category}
	Let $\cM$ be a simplicial model category and $\Phi$ a model filter with all elements cofibrant. Then we have an equivalence of categories $\Ho(\cM_\Phi) \simeq	\Ho(\cM)_\Phi$.
\end{proposition}

\begin{proof}
 By construction, we have the identity $\Ho(\cM) = \Ho(\Ho_\infty(\cM))$. By \cite[Theorem 2.22]{rasekh2021filterquotient}, the filter quotient $\infty$-category $\Ho_\infty(\cM)_\Phi$ is defined via a filtered colimit. So, the result follows from the fact that $\Ho\colon\cat_\infty \to \cat$ is the left adjoint to the inclusion and hence commutes with colimits \cite[Proposition 1.2.3.1]{lurie2009htt}.
\end{proof}

\section{The Case of Filter Products} \label{sec:filter products}
We now focus on a specific class of filter quotients, which provide many examples of interest: filter products. 

\begin{definition}
	Let $I$ be a set. A \emph{filter of subsets of $I$} is a filter on the poset $PI$, the power set of $I$.
\end{definition}

\begin{definition} 
	Let $\C$ be a category with an initial object $\emptyset$ and terminal object $1$, $I$ be a set, and $\Phi$ be a filter of subsets of $I$. Let $\Fil(\Phi)$ be the filter of subterminal objects in $\prod_I \C$, defined as 
	\[\Fil(\Phi) = \{ (1)_{i \in J} \colon J \in \Phi \}  \]
\end{definition}

If $\emptyset$ and $1$ are isomorphic, then $\Fil(\Phi)$ contains only one object and is hence trivial. We will hence assume $ \emptyset \not\cong 1$, in which case the assignment $\Fil$ is evidently injective. We will abuse notation and denote the filter of subterminal objects in $\prod_I \C$ by $\Phi$ again. 
For the filter to be homotopically relevant the initial object needs to be fibrant. Here we will focus on the case of \emph{strict initial objects}, i.e.,~initial objects with no non-trivial morphisms into them. This will guarantee fibrancy (it lifts against all trivial cofibrations) and not being isomorphic to the terminal object (except in the trivial category $[0]$).

\begin{lemma}
	Let $\cM$ be a model category with strict initial object, $I$ be a set and $\Phi$ a filter of subsets of $I$. Then the induced filter on $\prod_I \cM$ is a model filter.
\end{lemma}

\begin{proof}
	By \cref{ex:terminal object}, the filter is homotopical. Moreover, it is discrete by construction. The cofibrations and weak equivalences are $\Phi$-product stable, as the product with $J \in \Phi$, simply corresponds to projection $\prod_I \cM \to \prod_J \cM$.
\end{proof}

\begin{definition} \label{def:filter product model category}
	Let $\cM$ be a model category with strict initial object, $I$ be a set, and $\Phi$ a filter of subsets of $I$. Then the \emph{filter product model category} is the model category $(\prod_I \cM)_\Phi$, where $\Phi$ is the model filter in $\prod_I \cM$.
\end{definition}

\begin{notation}
	The filter product is commonly denoted by $\prod_\Phi \cM$.
\end{notation}

\begin{definition} \label{def:piphi}
	Let $\cM$ be a model category with strict initial object, $I$ a set, and $\Phi$ a filter of subsets on $I$. Let $\pi_\Phi\colon \cM \to \prod_\Phi \cM$ be defined as the composition 
	\[\cM \xrightarrow{ \ \Delta \ } \prod_I \cM \xrightarrow{ \ P_\Phi \ } \prod_\Phi \cM,\]
	where $\Delta$ is the diagonal functor.
\end{definition}

We have the following fact about $\pi_\Phi$. 

\begin{lemma}
Let $\cM$ be a model category with strict initial object, $I$ a set, and $\Phi$ a filter of subsets on $I$. Then the functor $\pi_\Phi\colon \cM \to \prod_\Phi \cM$ preserves weak equivalences, (co)fibrations and finite (co)limits.
\end{lemma}

Notice, the filter product construction will in fact preserve relevant properties of the original model structure. 

\begin{proposition} \label{prop:filter	product proper}
	Let $\cM$ be a model category with strict initial object, $I$ a set, and $\Phi$ a filter of subsets on $I$. If $\cM$ is (right or left) proper then so is $\prod_\Phi \cM$. 
\end{proposition}

\begin{proof}
	Let us assume $\cM$ is (right or left) proper. We want to show that $\prod_\Phi \cM$ is (right or left) proper. The case for right properness is immediate (\cref{prop:right proper}). The case for left	properness follows from the fact that $U \times -$ corresponds to the projection functor $\prod_I \cM \to \prod_U \cM$, which evidently preserves pushouts, and \cref{prop:left proper}. The case for properness follows from combining the last two steps.
\end{proof}

\begin{proposition} \label{prop:filter product enriched}
	Let $\cM$ be a $\V$-compactly enriched model category with strict initial object, $I$ a set, and $\Phi$ a filter of subsets on $I$. If $\V$ is quotient compatible, then $\Phi$ is a $\V$-enriched model filter, and hence $\prod_\Phi \cM$ is a $\V$-compactly enriched model category.
\end{proposition}

\begin{proof}
 We need to check that $\Phi$ is a $\V$-enriched model filter, the result will then follow from \cref{thm:enriched filter quotient}. Let $K$ be a compact object in $\V$ and $U$ an object in $\Phi$. Then $ - \times U\colon \prod_I \cM \to \prod_U \cM$ is simply the projection map, which evidently commutes with $- \otimes K\colon \cM \to \cM$. Hence, the filter quotient, in this case the filter product $\prod_\Phi \cM$, is $\V$-enriched.  
\end{proof}

We can in particular apply this to simplicial enrichment. 

\begin{corollary} \label{cor:filter product simplicial}
	 Let $\cM$ be a simplicial model category with strict initial object, $I$ a set, and $\Phi$ a filter on $I$. Then $\Phi$ is a simplicial model filter and $\prod_\Phi \cM$ is a simplicial model category. 
\end{corollary}

\begin{corollary} \label{cor:filter product Joyal}
	 Let $\cM$ be a model category enriched over the Joyal model structure with strict initial object, $I$ a set, and $\Phi$ a filter on $I$. Then $\Phi$ is a simplicial model filter and $\prod_\Phi \cM$ is a model category enriched over the Joyal model structure. 
\end{corollary}

We can now apply the result to Quillen adjunctions and equivalences.

\begin{corollary} \label{cor:quillen adjunction}
	Let $\cM, \cN$ be model categories with strict initial objects, $I$ a set, and $\Phi$ a filter of subsets on $I$. Let  
	\[
		\begin{tikzcd}
			\cM \arrow[r, "F", shift left = 1.8] \arrow[r, "\bot", leftarrow, shift right = 1.8, "G"'] & \cN 
		\end{tikzcd}
	\]
	be a Quillen	adjunction, such that $F$ preserves	the terminal object and $G$ preserves the initial object. Then this induces a Quillen adjunction 
	\[
		\begin{tikzcd}
			\prod_\Phi \cM \arrow[r, "\prod_I F", shift left = 1.8] \arrow[r, "\bot", leftarrow, shift right = 1.8, "\prod_I G"'] & \prod_\Phi \cN 
		\end{tikzcd}.
	\]
	Moreover, if $(F,G)$ is a homotopical (co)localization, Quillen equivalence, or enriched, then so is $(\prod_I F, \prod_I G)$.
\end{corollary}

\begin{proof}
 We need to check that the conditions of \cref{prop:quillen adjunction} are satisfied. Elements in $\Phi$ are by definition tuples of initial and terminal objects, which, by assumption, are preserved by $F$ and $G$. Moreover, $-\times U$ is simply a projection functor, which evidently commutes with $F$. Hence all conditions are satisfied and we are done.
\end{proof}

Finally, let us look at some examples. Recall that for a set $I$ and element $i$	in $I$, the \emph{principal filter} $\Phi_i$ at $i$ is given by $\{J \subseteq I \mid i \in J\}$. Note that $\Fil(\Phi_i)$ is a principal filter on $\prod_I \cM$, in the sense of \cref{ex:principalfilter}. The computation in the aforementioned example immediately implies the following.

\begin{example}
	Let $\cM$ be a model category with strict initial object, $I$ a set, and $i$ an element in $I$. Let $\Phi_i$ be the principal filter at $i$. Then the evaluation functor $\ev_i\colon \prod_{\Phi_i} \cM \to \cM$ is an equivalence.
\end{example}

\begin{example}
 Let $\cM$ be a model category with strict initial object, $I$ a set, and $J$ a subset of $I$. Let $\Phi_J$ be the filter consisting of subsets of $I$ that contain $J$. Then the evaluation functor $\ev_J\colon \prod_{\Phi_J} \cM \to \prod_J\cM$ is an equivalence.
\end{example}

\section{Examples} \label{sec:examples}
In this final section, we focus on examples of filter quotients. In the first two examples, we see that neither being cofibrantly generated nor combinatorial is in general preserved under filter products.

\begin{example} \label{ex:filter product sset}
	Let $\sset^{Kan}$ denote the category of simplicial sets with the Kan model structure, which is a proper simplicial model category \cite[Theorem II.3.1]{quillen1967modelcats}, that is also combinatorial \cite[Example 11.1.6]{hirschhorn2003modelcategories}. Let $\bN$ denote the set of natural numbers and $\cF$ the Fr{\'e}chet filter given by cofinite subsets of $\bN$. In this case the initial object in $\sset$ is indeed strict. Hence, the filter product $\prod_\cF \sset$ is a model category with the following properties: it is proper (\cref{prop:filter product proper}), simplicial (\cref{cor:filter product simplicial}), and Cartesian closed (\cref{prop:cartesian closure}).
	
	On the other hand, the category $\prod_\cF \sset$ does not have countable coproducts. Indeed, following the argument in \cite[Example D.5.1.7]{johnstone2002elephanti}, if $\prod_\cF \sset$ has infinite coproducts, then the object $(\bN)_{n \in \bN}$ has to be the countable coproduct of the terminal object, which is shown in that example not to be the case. Hence the model structure is neither combinatorial nor cofibrantly generated (\cref{def:cofibrantly generated combinatorial}).

	Finally, note that, while $\prod_\cF \sset$ is	not cofibrantly generated, it is still \emph{set-wise cofibrantly generated}, in the sense of \cite[Definition 2.9]{raptisrosicky2018smallpresentations}, by the set of (trivial) cofibrations 
	{\small
	\[ I =   \{(\partial \Delta[a_n])_{n \in \bN} \to (\Delta[a_n])_{n \in \bN} \mid (a_n)_{n \in \bN} \in \Hom_{\prod_\cF \sset}((1)_{n \in \bN},(\bN)_{n \in \bN}) \},\]
	\[ J = \{ ( \Lambda[a_n]_{i_n})_{n \in \bN} \to (\Delta[a_n])_{n \in \bN} \mid (a_n)_{n \in \bN}, (i_n)_{n \in \bN} \in \Hom_{\prod_\cF \sset}((1)_{n \in \bN},(\bN)_{n \in \bN}), (0)_{n \in \bN} \leq (i_n)_{n \in \bN} \leq (a_n)_{n \in \bN} \}.\]}
	Here set-wise cofibrantly generated means that a morphism is a (trivial) fibration if and only if it has the right lifting property with respect to all morphisms in $J$ (resp. $I$). While this condition is too weak for most aforementioned applications, such as the small object argument and the construction of new model structures, it can still be of independent interest.
\end{example}	

\begin{example} \label{ex:filter product top}
	Let $\Top^{Serre}$ denote the category of topological spaces with the Serre model structure, which is a proper simplicial model category \cite[Theorem II.3.1]{quillen1967modelcats}, that is cofibrantly generated \cite[Example 11.1.8]{hirschhorn2003modelcategories}, but not combinatorial \cite[Example 1.24 (7)]{adamekrosicky1994locallypresentable}. Again the empty topological space is strict, and so with $\bN$ and $\cF$ as in the previous example, $\prod_\cF \Top$ is a proper (\cref{prop:filter product proper}) and simplicial (\cref{cor:filter product simplicial}) model category. On the other hand, we can similarly see that $\prod_\cF \Top$ does not have countable coproducts and hence is not cofibrantly generated (\cref{def:cofibrantly generated combinatorial}).

	Again, similar to \cref{ex:filter product sset}, the model structure is set-wise cofibrantly generated by the set of (trivial) cofibrations
	\[ I =   \{(S^{a_n})_{n \in \bN} \to (D^{a_n+1})_{n \in \bN} \mid (a_n)_{n \in \bN} \in \Hom_{\prod_\cF \sset}((1)_{n \in \bN},(\bN)_{n \in \bN}) \},\]
	\[ J = \{ ( D^{a_n})_{n \in \bN} \to (D^{a_n} \times I)_{n \in \bN} \mid (a_n)_{n \in \bN} \in \Hom_{\prod_\cF \sset}((1)_{n \in \bN},(\bN)_{n \in \bN}) \}.\]
\end{example}	

For the next example, recall that an \emph{ultrafilter} of subsets of $I$ is a filter that is maximal as a subset of $PI$ with respect to inclusion. For a given element $i$ in $I$, the principal filter $\Phi_i = \{S \subseteq I \mid i \in S\} \subseteq PI$ is in particular an ultrafilter. All other ultrafilters are called \emph{non-principal}.

\begin{example} \label{ex:ultrafilter product sset}
	Let $\sset^{Kan}$ denote the category of simplicial sets with the Kan model structure. Let $\bN$ denote the set of natural numbers and let $\U$ be a non-principal ultrafilter on $\bN$. We then obtain a filter product model category $\prod_\U \sset$, which is simplicial. Notice, the global section functor $\Map((\Delta[0])_\bN,-)\colon \prod_\U \sset \to \sset$ is faithful, without being an equivalence, meaning the terminal object is a \emph{generator}, in the sense of \cite[Section VI.1]{maclanemoerdijk1994topos}. Indeed, it suffices to observe that the induced functor on homotopy categories 
	\[\Ho(\prod_\U \sset) \cong \prod_\U \Ho(\sset) \xrightarrow{ \pi_0\Map((\Delta[0])_\bN,-) } \Ho(\sset)\]
	is faithful. The first functor is just an equivalence, by \cref{prop:homotopy category}, and the second is faithful, by \cite[Example 9.45]{johnstone1977topos}.
\end{example}

Let us generalize these examples in two directions. 

\begin{definition}[\cite{rezk2010toposes}] \label{def:model topos}
	A \emph{model topos} is a left-exact Bousfield localization of the projective model structure on $\Fun(\C^{op},\sset)$, where $\C$ is a small category and $\sset$ has the Kan model structure.
\end{definition}

 The underlying category is a presheaf category, meaning the initial object is strict.

\begin{example}
  Let $\cG$ be a \emph{model topos}, $I$ a set and $\Phi$ a filter of subsets of $I$. Based on \cref{def:model topos}, we obtain a filter product model category $\prod_\Phi \cG$. 
\end{example}

\begin{example}
 Let $X$ be a topological space and $x$ a point in $X$. Let $\Phi_x$ be the filter of open subsets of $X$	containing $x$. Let $\Shv(X)$ be the model topos given as category of simplicial sheaves on the category $Open_X$. Then $\Shv(X)_{\Phi_x}$ has an induced model structure. 
\end{example}

Let us move on from examples of model categories modeling homotopy types (such as Kan complexes), to examples of model categories modeling $\infty$-categories.

\begin{example} \label{ex:joyal cosmos}
	Let $\sset^{Joy}$ denote the category of simplicial sets with the Joyal model structure, $I$ a set and $\cF$ a filter of subsets of $I$. Then, by \cref{cor:model category enriched joyal}, the filter product $\prod_\cF \sset^{Joy}$ is enriched over the Joyal model structure with all objects cofibrant. This means the full subcategory of fibrant objects is an $\infty$-cosmos \cite[Lemma 2.2.1]{riehlverity2017inftycosmos}. This in particular means there is a well-defined notion of \emph{Cartesian fibration}  \cite[Theorem 4.1.10]{riehlverity2017inftycosmos} that satisfies an internal version of the Yoneda lemma \cite[Theorem 6.0.1]{riehlverity2017inftycosmos}.

	If, in addition, $\cF$ is a non-principal ultrafilter, then we can repeat the same argument used in \cref{ex:ultrafilter product sset} to deduce that $\Hom((\Delta^0)_I,-)\colon \prod_\cF \sset \to \sset$ is faithful, meaning the terminal object is a generator. However, it is still the case that the category of fibrant objects is not locally presentable or accessible. Indeed, as the discrete simplicial set $\bN$ is a quasi-category, by \cref{thm:filter quotient model structure}, the object $(\bN)_{I}$ in $\prod_\cF \sset^{Joy}$ is also fibrant. Hence, following the same argument as in \cref{ex:filter product sset}, $\prod_\cF \sset^{Joy}$ cannot have countable coproducts.
	
	Hence, we have an example of an $\infty$-cosmos that is not locally presentable or accessible, in the sense of \cite{bourkelack2023accessiblecosmos}, that is still generated by the terminal object. This means we are generating a whole new class of $\infty$-cosmoi, not yet considered in the literature.
\end{example}

\begin{remark}
	Let $I$ be a set and $\cF$	a filter of subsets of $I$. The identity functor $\id\colon \sset^{Kan} \to \sset^{Joy}$ is right Quillen and homotopically fully faithful, and evidently preserves the initial object. Moreover, the left adjoint, which is also the identity, preserves the terminal object. This means, by \cref{cor:quillen adjunction}, we have an induced homotopically fully faithful right adjoint $\id\colon \prod_\cF \sset^{Kan} \to \prod_\cF \sset^{Joy}$. Intuitively, the image consists of the ``discrete $\infty$-categories'' or ``$\infty$-groupoids''.
\end{remark}

In these examples we focused on quasi-categories, however, the filter product construction is quite stable under change of models of $(\infty,1)$-categories. 

\begin{example}
	Let 
 \[
		\begin{tikzcd}
			\sset \arrow[r, "p_1^*", shift left = 1.8] \arrow[r, "\bot", leftarrow, shift right = 1.8, "i_1^*"'] & \sset^{\DD^{op}} \arrow[r, "t_!", shift left = 1.8] \arrow[r, "\bot", leftarrow, shift right = 1.8, "t^!"']  & \sset 
		\end{tikzcd} 
	\]
	denote the Quillen equivalences between the Joyal model structure on $\sset$ and the Rezk model structure on $\sset^{\DD^{op}}$ \cite{joyaltierney2007qcatvssegal}. Notice, both $p_1^*$ and $t_!$ preserve the terminal object and $i_1^*$ and $t^!$ preserve the initial object.
	
	Now, let $I$ be a set and $\cF$ be a filter of subsets. Then \cref{cor:quillen adjunction} implies that we have a diagram of Quillen equivalences
 \[
		\begin{tikzcd}
			\prod_\cF \sset \arrow[r, "p_1^*", shift left = 1.8] \arrow[r, "\bot", leftarrow, shift right = 1.8, "i_1^*"'] & \prod_\cF \sset^{\DD^{op}} \arrow[r, "t_!", shift left = 1.8] \arrow[r, "\bot", leftarrow, shift right = 1.8, "t^!"']  & \prod_\cF \sset 
		\end{tikzcd} 
	\]
	which we can think of as the filter product analogue of a change of model of $(\infty,1)$-categories between quasi-categories and complete Segal spaces. 
	
	We already observed above that $\prod_\cF \sset$ is an $\infty$-cosmos (\cref{ex:joyal cosmos}). We can similarly use the fact that the Rezk model structure gives us an $\infty$-cosmos \cite[Example 1.2.24]{riehlverity2022elements}. Now, by \cite[Corollary E.1.2]{riehlverity2022elements} this Quillen equivalence gives us a cosmological biequivalence between these $\infty$-cosmoi.
\end{example}

Let us finally observe a situation where filter quotients do not result in new examples.

\begin{example}
	Let $\cM$ be a pointed model category, meaning the unique map from the initial object to the terminal object is a weak equivalence. Then, following \cref{ex:pointed filter quotient}, every subterminal object is equivalent to the terminal object. Hence, every filter quotient is trivial.
\end{example}

\bibliographystyle{alpha}
\bibliography{main}

\end{document}